\documentclass[reqno]{amsart}
\usepackage{hyperref}
\usepackage{url}
\usepackage{amssymb}

\newtheorem{theorem}{Theorem}

\newtheorem{conj}{Conjecture}
\newtheorem{lemma}{Lemma}

\theoremstyle{definition}

\theoremstyle{remark}

\numberwithin{equation}{section}

\begin{document}

\title[Proof of a Conjecture on 6-colored Generalized Frobenius Partitions]
 {Proof of a Conjecture on 6-colored \\Generalized Frobenius Partitions}

\author{LIUQUAN WANG}
\address{Department of Mathematics, National University of Singapore, Singapore, 119076, SINGAPORE}

\email{wangliuquan@u.nus.edu; mathlqwang@163.com}



\date{July 6, 2015}
\dedicatory{}

\maketitle

\begin{abstract}
Let $c\phi_{k}(n)$ be the $k$-colored generalized Frobenius partition function. By employing the generating function of $c\phi_{6}(3n+1)$ found by Hirschhorn, we prove that $c\phi_{6}(27n+16)\equiv 0$ (mod 243). This confirms a conjecture of E.X.W. Xia. We also find a congruence relation $c\phi_{6}(81n+61) \equiv 3 c\phi_{6}(9n+7)$ (mod 243). Moreover, we show that $c\phi_{6}(81n+61) \equiv 0$ (mod 81), $c\phi_{6}(243n+142) \equiv 0$ (mod 243) and $c\phi_{6}(729n+ 547) \equiv 0$ (mod 243). We further conjecture that for $n\ge 0$, $c\phi_{6}(243n+142) \equiv 0$ (mod 729).
\end{abstract}

\section{introduction}
In his 1984 AMS Memoir, Andrews \cite{Andrews} introduced the concept of generalized Frobenius partitions. For any positive integer $k$, let $c\phi_{k}(n)$ denote the number of $k$-colored generalized Frobenius partition function of $n$. The generating function of $c\phi_{k}(n)$ is given by
\[\sum\limits_{n=0}^{\infty}{c\phi_{k}(n)q^n}=\frac{1}{(q;q)_{\infty}^k}\sum\limits_{m_1, \cdots, m_{k-1}=-\infty}^{\infty}{q^{Q(m_1,\cdots, m_{k-1})}},\]
where
\[Q(m_1,\cdots, m_{k-1})=\sum\limits_{i=1}^{k-1}{m_{i}^2}+\sum\limits_{1\le i <j \le k-1}{m_im_j}.\]

In particular, if $k=6$, Baruah and Sarmah \cite{Baruah} proved that
\begin{equation}\label{gen}
\begin{split}
\sum\limits_{n=0}^{\infty}{c\phi_{6}(n)q^n}&=\frac{1}{(q;q)_{\infty}^{6}}\Big(\varphi^3(q)\varphi(q^2)\varphi(q^6)+24q\psi^3(q)\psi(q^2)\psi(q^3)\\
&\quad \quad +4q^2\varphi^3(q)\psi(q^4)\psi(q^{12})\Big),
\end{split}
\end{equation}
where as usual, $\varphi(q)$ and $\psi(q)$ are Ramanujan's theta functions, namely (see \cite{Berndt}, for example)
\[\varphi(q)=\sum\limits_{n=0}^{\infty}{q^{n^2}}, \quad \psi(q)=\sum\limits_{n=0}^{\infty}{q^{n(n+1)/2}}.\]
They also established the 2- and 3- dissections of (\ref{gen}), from which they proved some interesting Ramanujan-type congruences: for $n\ge 0$,
\[c\phi_{6}(2n+1) \equiv 0 \pmod{4},\]
\begin{equation}\label{mod9a}
c\phi_{6}(3n+1) \equiv 0 \pmod{9},
\end{equation}
and
\begin{equation}\label{mod9b}
c\phi_{6}(3n+2) \equiv 0 \pmod{9}.
\end{equation}
They also prosed the following conjecture: for $n \ge 0$,
\begin{equation}\label{mod9}
c\phi_{6}(3n+2) \equiv 0 \pmod{27}.
\end{equation}

By utilizing the generating function for $c\phi_{6}(3n+2)$ given by Baruah and Sarmah \cite{Baruah} and the $(p,k)$-parametrization of theta functions due to Alaca and Williams, Xia \cite{Xia} proved (\ref{mod9}) and he further conjectured that
\begin{equation}\label{mod27}
c\phi_{6}(9n+7) \equiv 0 \pmod{27},
\end{equation}
and
\begin{equation}\label{mod243}
c\phi_{6}(27n+16)\equiv 0 \pmod{243}.
\end{equation}
By using some known $q$ series identities, Hirschhorn \cite{Hirschhorn} gave a new presentation of 3-dissections of (\ref{gen}), from which both (\ref{mod9a}) and (\ref{mod9b}) follows readily. He also proved (\ref{mod27}) but left (\ref{mod243}) still open.

In this note, by employing the generating function for $c\phi_{6}(3n+1)$ found by Hirschhorn \cite{Hirschhorn}, we obtain the following result.
\begin{theorem}\label{thm}
For any integer $n\ge 0$, we have
\begin{equation}\label{mod243a}
c\phi_{6}(27n+16) \equiv 0 \pmod{243},
\end{equation}
\begin{equation}\label{mod81}
c\phi_{6}(81n+61) \equiv 0 \pmod{81},
\end{equation}
\begin{equation}\label{mod243b}
c\phi_{6}(243n+142) \equiv 0 \pmod{243},
\end{equation}
\begin{equation}\label{mod243c}
c\phi_{6}(729n+ 547) \equiv 0 \pmod{243}
\end{equation}
and the congruence relation
\[c\phi_{6}(81n+61) \equiv 3 c\phi_{6}(9n+7) \pmod{243}.\]
\end{theorem}
Since
\[c\phi_{6}(16)=2 \times 3^5 \times 1222049, \quad c\phi_{6}(61)=2^2 \times 3^4 \times 19 \times 701612098458871\]
and
\begin{displaymath}
\begin{split}
c\phi_{6}(547)&=2^{5}\times 3^{5} \times 409\times 6661 \times 3949235117518927056389 \\
&\quad \times 20029030597437898896898971631,
\end{split}
\end{displaymath}
we see that (\ref{mod243a}), (\ref{mod81}) and (\ref{mod243c}) are best possible in the sense that the modulus cannot be replaced by higher powers of 3. However, based on some numerical evidences, we believe that (\ref{mod243b}) can be improved. We propose the following conjecture for research in the future.
\begin{conj}
For any integer $n \ge 0$, we have
\[c\phi_{6}(243n+142) \equiv 0 \pmod{729}.\]
\end{conj}

\section{Preliminaries}
In this section, we present some 3-dissection identities, which will play a key role in our proof.

It is easy to see that
\begin{displaymath}
\begin{split}
&\varphi(q)=\frac{(q^2;q^2)_{\infty}^{5}}{(q;q)_{\infty}^2(q^4;q^4)_{\infty}^2}, \quad \varphi(-q)=\frac{(q;q)_{\infty}^2}{(q^2;q^2)_{\infty}},\\
&\psi(q)=\frac{(q^2;q^2)_{\infty}^{2}}{(q;q)_{\infty}}, \quad \psi(-q)=\frac{(q;q)_{\infty}(q^4;q^4)_{\infty}}{(q^2;q^2)_{\infty}}.
\end{split}
\end{displaymath}
Moreover, let
\begin{displaymath}
\begin{split}
&X(q):=\sum\limits_{n=-\infty}^{\infty}{q^{3n^2+2n}}=\frac{(q^2;q^2)_{\infty}^{2}(q^3;q^3)_{\infty}(q^{12};q^{12})_{\infty}}{(q;q)_{\infty}(q^4;q^4)_{\infty}(q^6;q^6)_{\infty}},\\
&Y(q):=\sum\limits_{n=-\infty}^{\infty}{q^{n(3n+1)/2}}=\frac{(q^2;q^2)_{\infty}(q^3;q^3)_{\infty}^2}{(q;q)_{\infty}(q^6;q^6)_{\infty}}.
\end{split}
\end{displaymath}
We have the following 3-dissection identities (see \cite[Corollay, p.49]{notebook}):
\begin{equation}\label{3psi}
\varphi(q)=\varphi(q^9)+2qX(q^3),     \quad \psi(q)=Y(q^3)+q\psi(q^9).
\end{equation}

Let
\[a(q):=\sum\limits_{m,n=-\infty}^{\infty}{q^{m^2+mn+n^2}}=1+6\sum\limits_{n\ge 0}\Big(\frac{q^{3n+1}}{1-q^{3n+1}}-\frac{q^{3n+2}}{1-q^{3n+2}}\Big).\]
As shown in \cite{HGB}, we have
\[a(q)=a(q^3)+6q\frac{(q^9;q^9)_{\infty}^3}{(q^3;q^3)_{\infty}}.\]
This implies
\begin{equation}\label{aq5}
a(q)^5 \equiv a(q^3)^5+3qa(q^3)^4\frac{(q^9;q^9)_{\infty}^3}{(q^3;q^3)_{\infty}}+9q^2a(q^3)^3\frac{(q^9;q^9)_{\infty}^6}{(q^3;q^3)_{\infty}^2} \pmod{27},
\end{equation}
and
\begin{equation}\label{aq6}
a(q)^6 \equiv a(q^3)^6+9qa(q^3)^5\frac{(q^9;q^9)_{\infty}^3}{(q^3;q^3)_{\infty}} \pmod{27}.
\end{equation}
\begin{lemma}\label{basic}
Let $p$ be a prime and $\alpha $ be a positive integer. Then
\[(q;q)_{\infty }^{{{p}^{\alpha }}}\equiv ({{q}^{p}};{{q}^{p}})_{\infty }^{{{p}^{\alpha -1}}}  \pmod  {{{p}^{\alpha }}}.\]
\end{lemma}
\begin{proof}
Note that for any prime $p$, we have
\[\binom{p}{k} = \frac{p}{k}\cdot \binom{p-1}{k-1} \equiv 0 \pmod{p}, \quad 1 \le k \le p-1.\]
By the binomial theorem, we have
\[{{(1-x)}^{p}}=1-px+\cdots +p{{(-x)}^{p-1}}+{{(-x)}^{p}}\equiv 1-{{x}^{p}} \pmod{p}.\]
Hence we have
\[(q;q)_{\infty}^{p} \equiv (q^{p};q^{p})_{\infty} \pmod{p}.\]
This proves the lemma for $\alpha =1$. Suppose for some $\alpha \ge 2$ we have
\[(q;q)_{\infty}^{p^{\alpha-1}} \equiv (q^{p};q^{p})_{\infty}^{p^{\alpha-2}} \pmod{p^{\alpha-1}},\]
then there exists a series $f(q)$ with integer coefficients such that
\[(q;q)_{\infty}^{p^{\alpha-1}} = (q^{p};q^{p})_{\infty}^{p^{\alpha-2}} +p^{\alpha-1}f(q).\]
Again by the binomial theorem, we deduce that
\[(q;q)_{\infty}^{p^{\alpha}} =\Big((q^{p};q^{p})_{\infty}^{p^{\alpha-2}} +p^{\alpha-1}f(q)\Big)^{p} \equiv (q^{p};q^{p})_{\infty}^{p^{\alpha-1}} \pmod{p^{\alpha}}.\]
By induction on $\alpha$, we complete our proof.
\end{proof}

The following 3-dissection identity is also useful in our arguments.
\begin{lemma}\label{3dis}
We have
\begin{equation}\label{jacobi}
(q;q)_{\infty }^{3}=(q^3;q^3)_{\infty}a(q^3)-3q{{({{q}^{9}};{q}^{9}})_{\infty}^{3}}.
\end{equation}
\end{lemma}
\begin{proof}
By Jacobi's identity \cite[Theorem 1.3.9]{Berndt}, we have
\[(q;q)_{\infty }^{3}=\sum\limits_{n=0}^{\infty }{{{(-1)}^{n}}(2n+1){{q}^{n(n+1)/2}}}.\]
Note that $\frac{n(n+1)}{2}\equiv 0$ (mod 3) if and only if $n\equiv 0$ (mod 3) or $n\equiv 2$ (mod 3).
And $\frac{n(n+1)}{2}\equiv 1$ (mod 3) if and only if  $n\equiv 1$ (mod 3).
Hence we have the following 3-dissection identity
\[(q;q)_{\infty }^{3}=P({{q}^{3}})+qR({{q}^{3}}).\]
We have
\begin{displaymath}
\begin{split}
   P({{q}^{3}}) &=\sum\limits_{m=0}^{\infty }{{{(-1)}^{3m}}(6m+1){{q}^{3m(3m+1)/2}}}+\sum\limits_{m=0}^{\infty }{{{(-1)}^{3m+2}}(6m+5){{q}^{(3m+2)(3m+3)/2}}} \\
 & =\sum\limits_{m=0}^{\infty }{{{(-1)}^{m}}(6m+1){{q}^{3m(3m+1)/2}}}+\sum\limits_{m=-\infty }^{-1}{{{(-1)}^{m}}(6m+1){{q}^{3m(3m+1)/2}}} \\
 & =\sum\limits_{m=-\infty }^{\infty }{{{(-1)}^{m}}(6m+1){{q}^{3m(3m+1)/2}}}. \\
\end{split}
\end{displaymath}
Replacing ${{q}^{3}}$ by $q$, we obtain
\[P(q)=\sum\limits_{m=-\infty }^{\infty }{{{(-1)}^{m}}(6m+1){{q}^{m(3m+1)/2}}}.\]
From \cite{Pre} we know that
\[P(q)={{(q;q)}_{\infty }}\Bigg(1+6\sum\limits_{n\ge 0}{\Big(\frac{{{q}^{3n+1}}}{1-{{q}^{3n+1}}}-\frac{{{q}^{3n+2}}}{1-{{q}^{3n+2}}}\Big)}\Bigg)=(q;q)_{\infty}a(q).\]

Again, we have
\[qR({{q}^{3}})=\sum\limits_{m=0}^{\infty }{{{(-1)}^{3m+1}}(6m+3){{q}^{(3m+1)(3m+2)/2}}}.\]
Dividing both sides by $q$ and replacing ${{q}^{3}}$ by $q$, we deduce that
\[R(q)=-3\sum\limits_{m=0}^{\infty }{{{(-1)}^{m}}(2m+1){{q}^{3m(m+1)/2}}}=-3({{q}^{3}};{{q}^{3}})_{\infty }^{3}.\]
\end{proof}

\section{Proof of Theorem \ref{thm}}
From \cite{Hirschhorn} we find
\begin{equation}\label{3n1}
\begin{split}
&\quad \sum\limits_{n\ge 0}{c\phi_{6}(3n+1)q^n}\\
&=9\Bigg(\frac{(q^2;q^2)_{\infty}^5(q^3;q^3)_{\infty}^{6}}{(q;q)_{\infty}^{22}(q^4;q^4)_{\infty}^2}\bigg(2a(q)^5\frac{(q^3;q^3)_{\infty}^3}{(q;q)_{\infty}}+189qa(q)^2\frac{(q^3;q^3)_{\infty}^{12}}{(q;q)_{\infty}^4}\bigg)\\
&\quad +\frac{(q^3;q^3)_{\infty}^9(q^4;q^4)_{\infty}(q^6;q^6)_{\infty}^2}{(q;q)_{\infty}^{23}(q^2;q^2)_{\infty}(q^{12};q^{12})_{\infty}} \\
&\quad \times \bigg(2a(q)^6+378qa(q)^3\frac{(q^9;q^9)_{\infty}^{9}}{(q;q)_{\infty}^3}+1458q^2\frac{(q^3;q^3)_{\infty}^{18}}{(q;q)_{\infty}^6}\bigg)\\
&\quad -\frac{(q^3;q^3)_{\infty}^9(q^{12};q^{12})_{\infty}^2}{(q;q)_{\infty}^{23}(q^6;q^6)_{\infty}}\Big(36qa(q)^5\frac{(q^3;q^3)_{\infty}^3}{(q;q)_{\infty}}+1944q^2a(q)^2\frac{(q^3;q^3)_{\infty}^{12}}{(q;q)_{\infty}^4}\bigg)\Bigg).
\end{split}
\end{equation}
By Lemma \ref{basic} we have
\[(q;q)_{\infty}^3 \equiv (q^3;q^3)_{\infty} \pmod{3}, \quad (q;q)_{\infty}^{27}\equiv (q^3;q^3)_{\infty}^9 \pmod{27}. \]
Hence
\begin{equation}\label{start}
\begin{split}
&\quad \sum\limits_{n=0}^{\infty}{c\phi_{6}(3n+1)q^n}\\
\quad &\equiv 18\Bigg(a(q)^5\frac{(q^2;q^2)_{\infty}^5(q^3;q^3)_{\infty}^9}{(q;q)_{\infty}^{23}(q^4;q^4)_{\infty}^{2}} +a(q)^6\frac{(q^3;q^3)_{\infty}^9(q^4;q^4)_{\infty}(q^6;q^6)_{\infty}^2}{(q;q)_{\infty}^{23}(q^2;q^2)_{\infty}(q^{12};q^{12})_{\infty}}\\
&\quad -18qa(q)^5\frac{(q^3;q^3)_{\infty}^{12}(q^{12};q^{12})_{\infty}^2}{(q;q)_{\infty}^{24}(q^6;q^6)_{\infty}}\Bigg) \\
& \equiv 18\Bigg(a(q)^5\frac{(q^2;q^2)_{\infty}^5(q;q)_{\infty}^4}{(q^4;q^4)_{\infty}^2} +a(q)^6 \frac{(q;q)_{\infty}^4(q^4;q^4)_{\infty}(q^6;q^6)_{\infty}^2}{(q^2;q^2)_{\infty}(q^{12};q^{12})_{\infty}} \\
&\quad -18qa(q)^5\frac{(q^3;q^3)_{\infty}^{12}(q^{12};q^{12})_{\infty}^2}{(q^3;q^3)_{\infty}^8(q^6;q^6)_{\infty}}\Bigg) \pmod{243}. \\
\end{split}
\end{equation}

Note that
\begin{equation}\label{term1}
\frac{(q^2;q^2)_{\infty}^5(q;q)_{\infty}^4}{(q^4;q^4)_{\infty}^2}=\frac{(q^2;q^2)_{\infty}^5}{(q;q)_{\infty}^2(q^4;q^4)_{\infty}^2}\cdot (q;q)_{\infty}^6=\varphi(q)(q;q)_{\infty}^6,
\end{equation}
\begin{equation}\label{term2}
\frac{(q;q)_{\infty}^4(q^4;q^4)_{\infty}}{(q^2;q^2)_{\infty}}=\frac{(q;q)_{\infty}(q^4;q^4)_{\infty}}{(q^2;q^2)_{\infty}}\cdot (q;q)_{\infty}^3=\psi(-q)(q;q)_{\infty}^3.
\end{equation}

By (\ref{3psi}), (\ref{aq5}), (\ref{term1}) and Lemma \ref{3dis}, we have
\begin{displaymath}
\begin{split}
&\quad a(q)^5\frac{(q^2;q^2)_{\infty}^5(q;q)_{\infty}^4}{(q^4;q^4)_{\infty}^2} \\
&\equiv \Big(a(q^3)^5+3qa(q^3)^4\frac{(q^9;q^9)_{\infty}^3}{(q^3;q^3)_{\infty}} +9q^2a(q^3)^3\frac{(q^9;q^9)_{\infty}^6}{(q^3;q^3)_{\infty}^2}\Big) \Big(\varphi(q^9)+2qX(q^3)\Big)\\
&\quad \Big(a(q^3)^2(q^3;q^3)_{\infty}^2-6qa(q^3)(q^3;q^3)_{\infty}(q^9;q^9)_{\infty}^3+9q^2(q^9;q^9)_{\infty}^6\Big) \pmod{27},
\end{split}
\end{displaymath}
Extracting the terms of the form $q^{3n+2}$ in both sides, we obtain after simplification that
\begin{equation}\label{I}
q^2I(q^3) \equiv -6q^2a(q^3)^6X(q^3)(q^3;q^3)_{\infty}(q^9;q^9)_{\infty}^3 \pmod{27}.
\end{equation}
Thus
\[I(q)\equiv -6a(q)^6X(q)(q;q)_{\infty}(q^3;q^3)_{\infty}^3 \pmod{27}.\]

Similarly, by (\ref{3psi}), (\ref{aq6}), (\ref{term2}) and Lemma \ref{3dis} we have
\begin{displaymath}
\begin{split}
&\quad a(q)^6 \frac{(q;q)_{\infty}^4(q^4;q^4)_{\infty}(q^6;q^6)_{\infty}^2}{(q^2;q^2)_{\infty}(q^{12};q^{12})_{\infty}} \\
&\equiv \frac{(q^6;q^6)_{\infty}^2}{(q^{12};q^{12})_{\infty}} \Big(a(q^3)^6 +9qa(q^3)^5\frac{(q^9;q^9)_{\infty}^3}{(q^3;q^3)_{\infty}}\Big)\\
&\quad \Big(Y(-q^3)-q\psi(-q^9)\Big)\Big(a(q^3)(q^3;q^3)_{\infty}-3q(q^9;q^9)_{\infty}^3\Big) \pmod{27}.
\end{split}
\end{displaymath}
Extracting the terms of the form $q^{3n+2}$ in both sides, we obtain
\begin{equation}\label{J}
q^2J(q^3)\equiv -6q^2\frac{(q^6;q^6)_{\infty}^2}{(q^{12};q^{12})_{\infty}}a(q^3)^6\psi(-q^9)(q^9;q^9)_{\infty}^3 \pmod{27}.
\end{equation}
Thus
\[J(q) \equiv -6\frac{(q^2;q^2)_{\infty}^2}{(q^{4};q^{4})_{\infty}}a(q)^6\psi(-q^3)(q^3;q^3)_{\infty}^3 \pmod{27}.\]
From the product representations of $X(q)$ and $\psi(q)$, it is easy to see that $I(q)\equiv J(q)$ (mod 27).

In the same way, applying (\ref{aq5}) and extracting the terms of the form $q^{3n+2}$ in
\begin{displaymath}
-18qa(q)^5\frac{(q^3;q^3)_{\infty}^{12}(q^{12};q^{12})_{\infty}^2}{(q^3;q^3)_{\infty}^8(q^6;q^6)_{\infty}} \equiv -18qa(q^3)^5\frac{(q^3;q^3)_{\infty}^{4}(q^{12};q^{12})_{\infty}^2}{(q^6;q^6)_{\infty}} \pmod{27},
\end{displaymath}
we get
\[q^2K(q^3) \equiv 0 \pmod{27}.\]
Thus
\[K(q)\equiv 0 \pmod{27}.\]

If we extract the terms of the form $q^{3n+2}$ in (\ref{start}), divide by $q^2$ and replace $q^3$ by $q$, we deduce that
\begin{equation}\label{9n7}
\sum\limits_{n\ge 0}{c\phi_{6}(9n+7)q^n}\equiv 18\big(I(q)+J(q)+K(q)\big) \equiv 36J(q) \pmod{243}.
\end{equation}
Note that
\begin{equation}\label{Jdisc}
\begin{split}
J(q) &\equiv -6\varphi(-q^2)a(q^3)^6\psi(-q^3)(q^3;q^3)_{\infty}^3 \\
&\equiv -6\Big(\varphi(-q^{18})-2q^2X(-q^6)\Big)a(q^3)^6\psi(-q^3)(q^3;q^3)_{\infty}^3 \pmod{27}.
\end{split}
\end{equation}
Thus the terms of the form $q^{3n+1}$ in $J(q)$ vanish modulo 27. Therefore, by (\ref{9n7}) we see that
\[c\phi_{6}(27n+16) \equiv 0 \pmod{243}.\]

Furthermore, by (\ref{Jdisc}), extracting the terms of the form $q^{3n}$ in (\ref{9n7}), then replacing $q^3$ by $q$,  we obtain
\begin{displaymath}
\begin{split}
&\sum\limits_{n\ge0}{c\phi_{6}(27n+7)q^n} \\
&\equiv -216 \varphi(-q^6)a(q)^6\psi(-q)(q;q)_{\infty}^3 \\
& \equiv 27\varphi(-q^6)a(q^3)^6\Big(Y(-q^3)-q\psi(-q^9)\Big)\Big(a(q^3)(q^3;q^3)_{\infty}-3q(q^9;q^9)_{\infty}^3\Big) \pmod{243}.
\end{split}
\end{displaymath}
Extracting the terms of the form $q^{3n+2}$, dividing by $q^2$ and replacing $q^3$ by $q$, we get
\begin{equation}\label{mod81result}
\begin{split}
&\sum\limits_{n \ge 0}{c\phi_{6}(81n+61)q^n} \\
&\equiv 81 \varphi(-q^2)a(q)^6\psi(-q^3)(q^3;q^3)_{\infty}^3 \\
&\equiv 81 \Big(\varphi(-q^{18})-2q^2X(-q^6)\Big)a(q^3)^6\psi(-q^3)(q^3;q^3)_{\infty}^3 \pmod{243}.
\end{split}
\end{equation}
This implies (\ref{mod81}). Note that the terms of the form $q^{3n+1}$ do not appear in the right hand side of (\ref{mod81result}), we deduce that
\[c\phi_{6}(243n+142) \equiv 0 \pmod{243}.\]

From (\ref{9n7})--(\ref{mod81result}) we deduce that
\[c\phi_{6}(81n+61) \equiv 3 c\phi_{6}(9n+7) \pmod{243}.\]
Replacing $n$ by $9n+6$ in this congruence relation and applying (\ref{mod81}), we obtain
\[c\phi_{6}(729n+ 547) \equiv 0 \pmod{243}.\]


\begin{thebibliography}{0}
\bibitem{Andrews} G.E. Andrews, Generalized Frobenius Partitions, Mem. Amer. Math. Soc., vol.301, American Mathematical Society, Providence, RI, 1984.

\bibitem{Baruah}  N.D. Baruah and B.K. Sarmah, Generalized Frobenius partitions with 6 colors, Ramanujan
J. (2014), doi:10.1007/s11139-014-9595-2.

\bibitem{Berndt}B.C. Berndt, Number Theory in the Spirit of Ramanujan, Am. Math. Soc. Providence, 2006.

\bibitem{notebook} B. C. Berndt, Ramanujan's Notebooks. Part III, Springer-Verlag, New York, 1991.

\bibitem{HGB} M.D. Hirschhorn, F. Garvan, and J. Borwein, Cubic analogues of the Jacobian theta function $\theta(z, q)$, Can.
J. Math. \textbf{45} (1994), 673--694.

\bibitem{Hirschhorn} M.D. Hirschhorn, Some congruences for 6-colored generalized Frobenius partitions, Ramanujan J. (2015), doi:10.1007/s11139-015-9700-1.

\bibitem{Pre} M.D. Hirschhorn, Three classical results on representations of a number, The Andrews Festschrift (2001), 159--165.

\bibitem{Xia}E.X.W. Xia, Proof of a conjecture of Baruah and Sarmah on generalized Frobenius partitions with 6
colors, J. Number Theory \textbf{147} (2014), 852--860.

\end{thebibliography}
\end{document}